\documentclass[11pt, a4paper, reqno]{amsart}
\usepackage{a4wide}
\usepackage{dsfont}
\usepackage{amsmath, amsthm}
\usepackage{multirow}

\usepackage{color}
\usepackage{ifpdf}
\ifpdf %if using pdfLaTeX in PDF mode
  \usepackage[pdftex]{graphicx}
  \DeclareGraphicsExtensions{.pdf,.png,.mps}
  \usepackage{pgf}
  \usepackage{tikz}
\else %if using LaTeX or pdfLaTeX in DVI mode
  \usepackage{graphicx}
  \DeclareGraphicsExtensions{.eps,.bmp}
  \usepackage{pgf}
\fi
\usepackage{epic}
\usepackage{wrapfig}% package for wrapfigure environment

\newtheorem{thm}{Theorem}%[section]

\newtheorem{lem}[thm]{Lemma}

\newtheorem{conj}[thm]{Conjecture}
\theoremstyle{definition}

\newtheorem*{rmk}{Remark}

\newcommand\set[1]{\left\{#1\right\}}

\newcommand\floor[1]{\left\lfloor#1\right\rfloor}

\def\S{\mathfrak{S}}

\DeclareMathOperator\exc{exc}

\DeclareMathOperator\drop{drop}
\DeclareMathOperator\defi{drop}

\DeclareMathOperator\des{des}

\DeclareMathOperator\cros{cros}
\DeclareMathOperator\nest{nest}

\DeclareMathOperator\inv{inv}
\mathchardef\mhyphen="2D
\DeclareMathOperator\les{(31\mhyphen 2)}
\DeclareMathOperator\less{(13\mhyphen 2)}
\DeclareMathOperator\res{(2\mhyphen 13)}
\DeclareMathOperator\ress{(2\mhyphen 31)}

\DeclareMathOperator\LES{31\mhyphen 2}

\DeclareMathOperator\RESS{2\mhyphen 31}

\def\312{\les}
\def\132{\less}
\def\213{\res}
\def\231{\ress}
\title[An expansion formula for the inversions and excedances in the symmetric group]{An expansion formula for the inversions and excedances in the symmetric group}
\author{Jiang Zeng}
\address[Jiang Zeng]{Universit\'{e} de Lyon; UMR 5208 du CNRS; Universit\'{e} Lyon 1; Institut Camille Jordan;  43, blvd du 11 novembre 1918, F-69622 Villeurbanne Cedex, France}
\email{zeng@math.univ-lyon1.fr}

\date{\today}
\begin{document}
\maketitle

\begin{abstract}
We prove a recent conjecture of  Blanco and Petersen (arXiv:1206.0803v2)
about an expansion formula for inversions and excedances in the symmetric group.
\end{abstract}

\section{Introduction}
Let $\S_n$ be the set of permutations  on $[n]=\set{1,\dots,n}$. For $\sigma=\sigma(1)\dots\sigma(n)\in \S_n$, 
the entry $i \in [n]$  is called an \emph{excedance} (resp.  \emph{drop})  of $\sigma$ if $i < \sigma(i)$ (resp.  $i > \sigma(i)$);
 the entry $i\in [n]$ is called a \emph{descent} of $\sigma$ if $i<n$ and $\sigma(i)>\sigma(i+1)$. 
Denote the number of descents, excedances and drops in $\sigma$ by $\des \sigma$, $\exc \sigma$ and  $\drop\sigma$, respectively.
It is well known \cite{FS70} that the  statistics $\des$, $\exc$ and $\drop$  have the same distribution on~$\S_n$, 
 their common enumerative polynomial is called 
the Eulerian polynomial:
\begin{align}\label{eq:eulerian}
A_n(t)=\sum_{\sigma \in \S_n} t^{\des\sigma} =\sum_{\sigma \in \S_n} t^{\exc\sigma} = \sum_{\sigma \in \S_n} t^{\defi\sigma},
\end{align}
 and there are nonnegative integers $\gamma_{n,j} $ such that
\begin{align}\label{eq:peaka}
A_n(t) = \sum_{k=0}^{\lfloor (n-1)/2\rfloor} \gamma_{n,k} t^{k} (1+t)^{n-1-2k}.
\end{align}

In the last years several $q$-analogues of \eqref{eq:peaka} were established in  \cite{SW10,HJZ12, SZ12} by combining one of the Eulerian statistics in \eqref{eq:eulerian} 
and the Major index.
Let $\inv\sigma$ be the number of inversions in $\sigma\in \S_n$. The  following $q$-Eulerian polynomial
\begin{align}\label{eq:qeuler}
S_n(q,t)=\sum_{\sigma\in \S_n}q^{\inv\sigma}t^{\exc \sigma}
\end{align}
was first studied in \cite{CSZ97}.  Recently, 
  Blanco and Petersen~\cite[Conj. 3.1]{BP12} has  conjectured that  $S_n(q,t)$
would display a nice  $q$-version of  \eqref{eq:peaka}.
\begin{conj}[Blanco and Petersen]
There exist polynomials $\gamma_{n,k}(q)$ with nonegative integer coefficients such that:
\begin{align}
S_n(q,t/q):=\sum_{\sigma\in \S_n}q^{\inv\sigma-\exc\sigma}t^{\exc \sigma}=\sum_{0\leq k\leq (n-1)/2}\gamma_{n,k}(q) t^k(1+t)^{n-1-2k}.
\end{align}
\end{conj}

The aim of this paper is to show that   a stronger version of the above conjecture was implicitly given  in  \cite{SZ10,SZ12}.

\section{Prelimaries}
For  $\sigma \in \S_n$,
the statistic $\les \sigma$ (resp.  $\less \sigma$ )
 is the number of pairs $(i,j)$ such that $2\leq i<j\leq n$ and $ \sigma(i-1)>\sigma(j)>\sigma(i)$
 (resp.  $ \sigma(i-1)<\sigma(j)<\sigma(i)$).
Similarly,
the statistic $\res\sigma $ (resp. $\ress \sigma $)
is the number of pairs $(i,j)$ such that $1\leq i<j\leq n-1$ and $\sigma(j+1)>\sigma(i)>\sigma(j)$ (resp.
$\sigma(j+1)<\sigma(i)<\sigma(j)$).
Define  the generalized  Eulerian polynomial  by
\begin{align}\label{eq:dfA}
A_n(p,q,t) := \sum_{\sigma\in \S_n} p^{\res \sigma} q^{\les \sigma}  t^{\des\sigma}.
\end{align}

For  $\sigma \in \S_n$,   the value  $\sigma(i)$,  for $i\in [n]$,   is called 
\begin{itemize}
\item a \emph{valley}  if $\sigma(i-1) > \sigma(i)$ and $\sigma(i) < \sigma(i+1)$;
\item a \emph{double descent} if $\sigma(i-1) > \sigma(i)$ and $\sigma(i) > \sigma(i+1)$.
\end{itemize}
where   $\sigma(0)=\sigma(n+1)=0$ by convention.
Let $\S_{n,k}$  be  the subset of permutations $\sigma\in \S_{n}$ with exactly $k$ valleys and without double descents.
Define the polynomial
\begin{align} \label{eq:combin}
a_{n,k}(p,q)=\sum_{\sigma\in \S_{n,k}} p^{\res \sigma} q^{\les \sigma}.
\end{align}
The following lemma is a special case of Theorem~2 in \cite{SZ12}.
\begin{lem}\label{thm:a}
We have the expansion formula
\begin{equation}\label{eq:def}
A_n(p,q,t)= \sum_{k=0}^{\floor{(n-1)/2}} a_{n,k}(p,q) t^k (1+t)^{n-1-2k}.
\end{equation}
Moreover, for all $0\le k \le \floor{(n-1)/2}$, the following divisibility holds
\begin{align}\label{eq:conj}
(p+q)^k  ~|~ a_{n,k}(p,q).
\end{align}
\end{lem}

For a permutation $\sigma\in \S_{n}$ the crossing and nesting numbers are defined by
\begin{align}
 \cros\sigma&= \# \set{(i,j)\in [n]\times[n]: (i<j\le\sigma(i)<\sigma(j))\vee(i>j>\sigma(i)>\sigma(j))},\\
\nest\sigma &= \# \set{(i,j)\in [n]\times[n]: (i<j\le\sigma(j)<\sigma(i))\vee(i>j>\sigma(j)>\sigma(i))}.
\end{align}
For example,  if the permutation is  $\sigma=3762154=\left({1 \atop 3}{2 \atop 7}{3 \atop 6}{4 \atop 2}{5 \atop 1}{6 \atop 5}{7 \atop 4}\right)$, 
 there are three crossings $(1<2<\sigma(1)<\sigma(2))$, $(1<3\le \sigma(1)<\sigma(3))$, $(7>4>\sigma(7)>\sigma(4))$ and three nestings $(2<3<\sigma(3)<\sigma(2))$, $(5>4>\sigma(4)>\sigma(5))$, $(7>6>\sigma(6)>\sigma(7))$, thus $\cros\sigma=\nest\sigma=3$.

The following lemma is a special case of Theorem~5 in \cite{SZ12}.

%%%%%%%%%%%%
\begin{lem}\label{thm3}
There is a bijection $\Phi$ on $\S_n$ such that  for all $\sigma\in \S_n$ we have
\begin{align*}
(\nest , \cros, \defi)\sigma =(\RESS, \LES, \des)\Phi(\sigma).
\end{align*}
\end{lem}

The following lemma is proved in \cite[Eq. (39)]{SZ12} using continued fractions of their  generating functions.
\begin{lem}
The triple statistics $(\ress, \les, \des)$ and $(\res, \les, \des)$ are equidistributed on $\S_n$.
\end{lem}

\section{Proof of the conjecture}
We first state  a more precise version of Conjecture~1.
\begin{thm} The Conjecture 1 holds true. Moreover, for $0\leq k\leq (n-1)/2$,
$$
\gamma_{n,k}(q)=a_{n,k}(q^2,q)=\sum_{\sigma\in \S_{n,k}} q^{2\res \sigma+\les \sigma}
$$
and 
$\gamma_{n,k}(q)$ is divisible by $q^k(1+q)^k$.
\end{thm}
\begin{proof}
Since  $\inv\sigma^{-1}=\inv\sigma$ and $\drop\sigma=\exc\sigma^{-1}$ for any $\sigma\in \S_n$, we have 
$$
S_n(q,t)=\sum_{\sigma\in \S_n}q^{\inv\sigma}t^{\exc \sigma}=\sum_{\sigma\in \S_n}q^{\inv\sigma}t^{\drop\sigma}.
$$
Combining Lemmas~3 and 4 yields
\begin{align*}
A_n(p,q,t) =
\sum_{\sigma\in \S_n} p^{\nest \sigma} q^{\cros \sigma} t^{\defi\sigma}.
\end{align*}
Therefore,  as $\inv=\drop+\cros+2\nest$ (see \cite[Eq. (40)]{SZ10}),  we deduce
$$
S_n(q,t/q)=A_n(q^2,q,t).
$$
The desired result follows then from Lemma~2.
\end{proof}

\begin{rmk}
We conclude this paper with the following open problems:
\begin{itemize}
\item[1)] Is there a simple biective proof of Lemma~4? 
\item[2)] Is there a type $B$ analogue of  (2) for inversions and excedances?
\end{itemize}
\end{rmk}
\providecommand{\bysame}{\leavevmode\hbox to3em{\hrulefill}\thinspace}
\providecommand{\href}[2]{#2}

% ----------------------------------------------------------------
\end{document}